\documentclass[a4paper, 12pt, final]{amsart}
\usepackage{}
\usepackage[OT1]{fontenc}
\usepackage[latin1]{inputenc}
\usepackage{amssymb,latexsym}
\usepackage{type1cm,ae}
\usepackage[notref,notcite]{showkeys}
\usepackage{graphicx}
\usepackage{xspace}
\usepackage{setspace}
\usepackage[english]{babel}
\usepackage{tikz}

\theoremstyle{plain}		\newtheorem{theorem}{Theorem}[section]

				\newtheorem{corollary}[theorem]{Corollary}

\numberwithin{equation}{section}

\newcommand*{\bR}{\ensuremath{\mathbb{R}}}

\newcommand*{\bdary}[1]{\partial #1}

\newcommand*{\Wert}{\mathord{\mbox{|\kern-1.5pt|\kern-1.5pt|}}}
\newcommand*{\ie}{\mbox{i.e.}\xspace}

\DeclareMathOperator{\diam}{diam}

\def\XXint#1#2#3{{\setbox0=\hbox{$#1{#2#3}{\int}$}
  \vcenter{\hbox{$#2#3$}}\kern-.5\wd0}}

\title[Sharpness of continuity of quasiconformal mappings]
{Sharpness of uniform continuity of quasiconformal
mappings onto s-John domains}

\author{Chang-yu Guo}
\address[Chang-Yu Guo]{Department of Mathematics and Statistics, University of Jyv\"askyl\"a, P.O. Box 35, FI-40014 University of Jyv\"askyl\"a, Finland}
\email{changyu.c.guo@jyu.fi}

\author{Pekka Koskela}
\address[Pekka Koskela]{Department of Mathematics and Statistics, University of Jyv\"askyl\"a, P.O. Box 35, FI-40014 University of Jyv\"askyl\"a, Finland}
\email{pkoskela@maths.jyu.fi}

\thanks{C.Y.Guo and P.Koskela were partially supported by the Academy of Finland grant 131477.}

\begin{document}
\begin{abstract}
We construct examples to show the sharpness of uniform continuity of quasiconformal mappings onto $s$-John domains. Our examples also give a negative answer to a prediction in~\cite{hk05}.
\end{abstract}

\maketitle

\section{Introduction}\label{sec:first}
Recall that a bounded domain $\Omega\subset \bR^n$ is a John domain if there
is a constant $C$ and a point $x_0\in \Omega$ so that, for each $x\in \Omega,$
one can find a rectifiable curve $\gamma:[0,1]\to \Omega$ with $\gamma(0)=x,$
$\gamma(1)=x_0$ and with
\begin{equation} \label{eka}
    Cd(\gamma(t),\bdary\Omega)\ge l(\gamma([0,t]))
\end{equation}
for each $0<t\le 1.$ F. John used this condition in his work on
elasticity~\cite{j61} and the term was coined by
Martio and Sarvas~\cite{ms79}. Smith and Stegenga \cite{ss90} introduced
the more general concept of $s$-John domains, $s\ge 1,$ by replacing \eqref{eka}
with
\begin{equation} \label{toka}
    Cd(\gamma(t),\bdary\Omega)\ge l(\gamma([0,t]))^s.
\end{equation}
The recent studies \cite{aim09,g13b,gkt12} on mappings of finite distortion
have generated new interest in the class of $s$-John domains.


In this paper, we are interested in uniform continuity of those quasiconformal
mappings whose target domain is $s$-John. For the $s=1$ case,
one always has uniform H\"older continuity:
\begin{equation}\label{eq:holder continuity}
|f(x)-f(y)|\leq C d_I(x,y)^{\alpha},
\end{equation}
provided $f:\Omega'\to \Omega$ is a quasiconformal mapping, see \cite{kot01}.
Here $\alpha$ depends on the constant in the 1-John condition for $\Omega,$
on the quasiconformality constant of $f,$ and on the underlying dimension.
The internal distance $d_I(z,w)$ for a pair of points in a domain $G$ is
the infimum of the lengths of all paths that join $z$ to $w$ in $G.$

In~\cite{g13}, the following uniform continuity result for general
$s$-John domains was established.

\begin{theorem}\label{thm:modulus of continuity in higher dimension}
Let $\Omega'\subset\bR^n$ be a domain and $\Omega\subset\bR^n$ be an $s$-John domain
with $s\in (1,1+\frac{1}{n-1})$. Then each quasiconformal mapping
$f:\Omega'\to\Omega$ satisfies
\begin{equation}\label{eq:modulus of continuity in higher dimension}
    D_I(f(x'),f(y'))\leq C\Big(\log\frac{1}{Cd_I(x',y')}\Big)^{-\frac{1}{s-1}}
\end{equation}
for every pair $x',y'$ of distinct points in $\Omega'$, where $D_I$ is
defined by
taking the infimum of the diameters over all rectifiable curves in $\Omega$
joining the desired pair of points.
\end{theorem}
Notice that for all $z,w\in \Omega$,
\begin{align*}
|z-w|\leq D_I(z,w);
\end{align*}
hence the left-hand side of~\eqref{eq:modulus of continuity in higher dimension} can be replaced with $|f(x')-f(y')|$. It was shown in \cite{g13} that the restriction $s<1+\frac{1}{n-1}$ can
be disposed of if $\Omega$ equipped with the quasihyperbolic metric is Gromov
hyperbolic.

Our first result of this paper shows that the requirement that $s<1+\frac{1}{n-1}$ in
Theorem~\ref{thm:modulus of continuity in higher dimension} is sharp when $n=2$. This is somewhat surprising since~\eqref{eq:modulus of continuity in higher dimension} does not degenerate when $s=1+\frac{1}{n-1}$.

\begin{theorem}\label{example:sharpness}
There exist a bounded domain
$\Omega'\subset\bR^2$, a $2$-John domain $\Omega\subset\bR^2$,
and a quasiconformal mapping 
$f:\Omega'\to \Omega$ such that $f$ is not 
uniformly continuous with respect to the metrics $d(x,y)=|x-y|$ in $\Omega$ and $d_I$ in $\Omega'$.
\end{theorem}

Since each simply (or finitely) connected planar domain is Gromov hyperbolic
when equipped with the quasihyperbolic metric, the $2$-John domain above
is necessarily infinitely connected.

Actually, the results in~\cite{kot01} establish~\eqref{eq:holder continuity} under more general setting than the case of $1$-John domains. Indeed, it was proven there that it suffices to assume that
\begin{equation}\label{eq:kot01}
k_\Omega(x,x_0)\leq C_1\log\frac{1}{d(x,\bdary\Omega)}+C_2
\end{equation}
for some constants $C_1$, $C_2$ and a fixed point $x_0\in \Omega$, where
\begin{align*}
k_\Omega(x,x_0)=\inf_{\gamma_{x}}\int_{\gamma_{x}}\frac{ds}{d(z,\bdary\Omega)}
\end{align*}
is the quasihyperbolic distance between $x$ and $x_0$; the infimum is taken over all rectifiable curves $\gamma_x$ in $\Omega$ which join $x$ to $x_0$. For $x,y\in\Omega$, there is a (quasihyperbolic) geodesic $[x,y]$ in $\Omega$ with
\begin{align*}
k_\Omega(x,y)=\int_{[x,y]}\frac{ds}{d(z,\bdary\Omega)},
\end{align*}
see~\cite{go79}.
In~\cite{hk05},~\eqref{eq:kot01} was further replaced with
\begin{align*}
k_\Omega(x,x_0)\leq \phi\Big(\frac{1}{d(x,\bdary\Omega)}\Big),
\end{align*}
under the assumption that
\begin{align*}
\int_1^\infty \frac{dt}{\phi^{-1}(t)}<\infty.
\end{align*}
A uniform continuity estimate of the type~\eqref{eq:holder continuity} was established under the additional assumption that $t\mapsto \Phi(t)^{-a}$ is concave for some $a>n-1$, where
\begin{align*}
\Phi(t)=\psi^{-1}(t)\quad\text{and}\quad \psi(t)=\int_t^\infty \frac{ds}{\phi^{-1}(s)}.
\end{align*}
This concavity assumption was speculated in~\cite{hk05} to be superfluous. Our construction refutes this speculation.

\begin{corollary}\label{coro:co-example to hk}
There exist a bounded domain
$\Omega\subset\bR^2$ and a point $x_0\in \Omega$ such that
\begin{equation}\label{eq:QH growth}
k_\Omega(x,x_0)\leq Cd(x,\bdary\Omega)^{-\frac{1}{2}}
\end{equation}
for all $x\in \Omega$ and a quasiconformal mapping $f:\Omega'\to \Omega$, where $\Omega'\subset\bR^2$,  such that $f$ is not uniformly continuous with respect to the metrics $d(x,y)=|x-y|$ in $\Omega$ and $d_I$ in $\Omega'$.
\end{corollary}



Our next result shows that Theorem~\ref{thm:modulus of continuity in higher dimension} is essentially sharp for all dimensions.
\begin{theorem}\label{example:sharpness2}
Let $n\geq 3$. There exist a bounded domain
$\Omega'\subset\bR^n$, a domain $\Omega\subset\bR^n$ that is 
$s$-John for any $s\in (1+\frac{1}{n-1}, \infty)$,
and a quasiconformal mapping
 $f:\Omega'\to \Omega$ such that 
 $f$ is not uniformly continuous with respect to
  the metrics $d(x,y)=|x-y|$ in $\Omega$ and $d_I$ in $\Omega'$.
\end{theorem}
It would be interesting to know whether one can allow for $s=1+\frac{1}{n-1}$ in Theorem~\ref{example:sharpness2}.

As a by-product of our construction, we obtain the following corollary, which implies that the concavity condition mentioned above is necessary in all dimensions.
\begin{corollary}\label{coro:co-example to hk2}
Let $n\geq 3$. There exist a bounded domain
$\Omega\subset\bR^n$ and a point $x_0\in \Omega$ such that
\begin{equation}\label{eq:QH growth2}
k_\Omega(x,x_0)\leq Cd(x,\bdary\Omega)^{-\frac{1}{n}}\log\frac{C}{d(x,\bdary\Omega)}
\end{equation}
for all $x\in \Omega$ and a quasiconformal mapping $f:\Omega'\to \Omega$, where $\Omega'\subset\bR^n$,  such that $f$ is not uniformly continuous with respect to the metrics $d(x,y)=|x-y|$ in $\Omega$ and $d_I$ in $\Omega'$.
\end{corollary}

In Theorem~\ref{example:sharpness2}, the mapping $f$ actually satisfies
\begin{equation}\label{eq:internal uniform continuity}
D_I(f(x'),f(y'))\leq C\Big(\log\frac{1}{D_{QH}(x',y')}\Big)^{-\frac{1}{s-1}},
\end{equation}
where $D_{QH}(x,y)=\diam [x,y]$. Notice that $D_{QH}(x',y')\geq D_I(x',y')$ and that $D_{QH}(x',y')\leq Cd_I(x',y')$ if $\Omega'$ satisfies a Gehring-Hayman inequality, especially if $\Omega'$ is Gromov hyperbolic when equipped with the quasihyperbolic metric~\cite{bhk01}.

One could then hope that Theorem~\ref{thm:modulus of continuity in higher dimension} extends to hold for all $s>1$ under the Gehring-Hayman assumption and even that~\eqref{eq:internal uniform continuity} holds for all $s>1$. This turns out not to be the case, even though a weaker version of~\eqref{eq:internal uniform continuity} does indeed follow from the results in~\cite{hk05} when $1<s<2$.

\begin{theorem}\label{example:sharp3}
Let $s\in (2,\infty)$ and $n\geq 2$. There exist a bounded domain
$\Omega'\subset\bR^n$, an $s$-John domain $\Omega\subset\bR^n$, both satisfying the Gehring-Hayman inequality, 
and a quasiconformal mapping $f:\Omega'\to \Omega$ such that $f$ fails to satisfy~\eqref{eq:internal uniform continuity}. 
\end{theorem}

It would be interesting to know whether one can allow for $s=2$ in Theorem~\ref{example:sharp3}. 


\section{Proofs of the main results}

\begin{proof}[Proof of Theorem~\ref{example:sharpness}]

Our 2-John domain $\Omega$ will be constructed inductively as indicated in Figure~\ref{fig:domain1}.
\begin{figure}[h]
  \includegraphics[width=10.5cm]{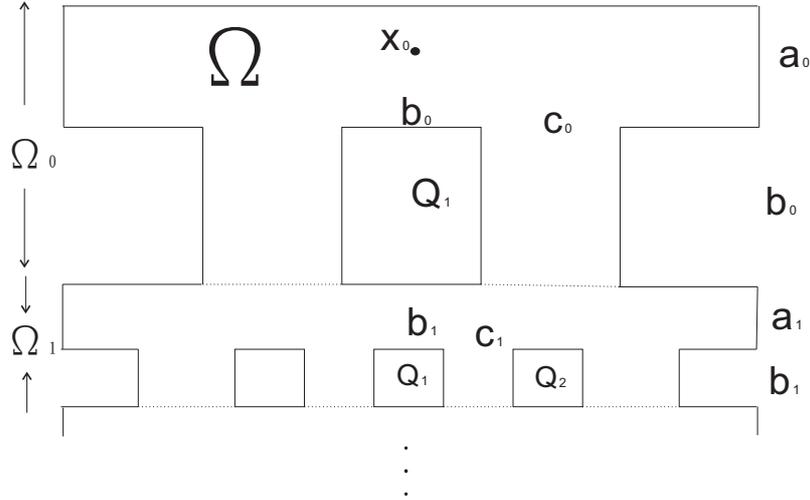}\\
  \caption{the 2-John domain $\Omega$}\label{fig:domain1}
\end{figure}

Set $a_j=2^{-2(j+1)}, b_j=2^{-j}$ and $c_j=2^{-2(j+1)}$. For $j=0$, we let the $\Omega_0$-part consist of a rectangle of length 1 and width $a_0$ and two rectangular ``legs" of width $c_0$ and length $b_0$. The two rectangular ``legs" are obtained in the following manner: first remove the central square $Q_1$ of side-length $b_0$; then set the distance between $Q_1$ and the vertical boundary of $\Omega_0$ to be $c_0$. Next, for $j=1$, we let the $\Omega_1$-part consist of a rectangle of length 1 and width $a_1$ and four rectangular ``legs" of width $c_1$ and length $b_1$. The four rectangular ``legs" are obtained in a similar fashion as before: first remove 3 squares of side-length $b_1$; then make them equi-distributed, \ie the gap between two consecutive squares is $c_1$; finally set the distance between $Q_2$ and the vertical boundary of  $\Omega_1$ to be $c_1$. We continue the process. Let the $\Omega_j$-part consist of a rectangle of length 1 and width $a_j$ and $2^j$ rectangular ``legs" of width $c_j$ and length $b_j$. The rectangular ``legs" are obtained by removing $2^{j+1}-1$ equi-distributed squares of side-length $b_j$ in a similar way as before. Among these removed squares, we label from middle to the right-most as $Q_1$, $Q_2$, $\dots$, $Q_{2^j}$ respectively. According to our construction, the distance between two consecutive removed squares is $c_j$ and the distance between $Q_{2^j}$ and the vertical boundary of $\Omega_j$ is also $c_j$. Finally, our domain $\Omega$ is the union of all $\Omega_j$'s. It is clear from the construction that $\Omega$ is 2-John and symmetric with respect to the $y$-axis.

We next construct our source domain $\Omega'$ and a quasiconformal mapping $g:\Omega'\to \Omega$, which is not uniformly continuous with respect to the metrics $d(x,y)=|x-y|$ in $\Omega$ and $d_I$ in $\Omega'$. Actually, we construct a quasiconformal mapping $f:\Omega\to\Omega'$ whose (quasiconformal) inverse has the desired properties.

The idea is demonstrated in Figure~\ref{fig:domain2}: we scale the upper part of each $\Omega_j$ by $\frac{1}{j}$ and replace the associated $2^{j+1}$ rectangular ``legs" by the same number of new ``legs". The vertical distance between the scaled upper parts of $\Omega_j$ and $\Omega_{j+1}$ is set to be $2j^{-2}$.
We also make the domain $\Omega'$ symmetric with respect to $y$-axis. 
Since the distance between two consecutive 
legs in $\Omega_j$ is  $2^{-j}$, the distance between the tops of two 
consecutive ``legs" in $\Omega_j'$ is $\frac{2^{-j}}{j}$. For the bottoms, the distance is approximately $\frac{2^{-j-1}}{j+1}$.
\begin{figure}[h]
  \includegraphics[width=12cm]{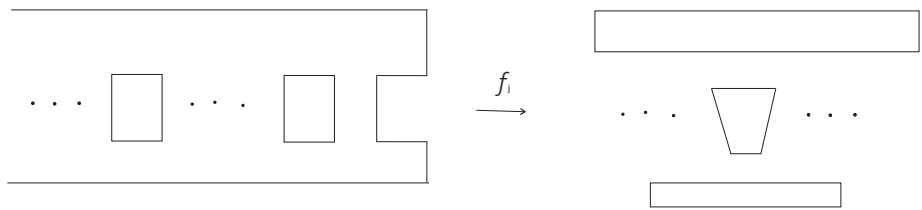}\\
  \caption{$\Omega$ and $\Omega'$ in the step $j$}\label{fig:domain2}
\end{figure}

Recall the labelled squares $Q_i$, $i=1,\dots,2^j$ introduced in $\Omega_j$. We denote by $\tilde{Q}_i$ the ``leg" next to $Q_i$, on the right.  We will construct a quasiconformal mapping $f_j$ from the (translated) rectangle $\tilde{Q}_i$ to the (translated) new ``leg" $Q_i'$ as in Figure~\ref{fig:domain3}. $Q_i'$ consists of two parts $A'$ and $B'$. The distance between the bottom line segment $0\textbf{a}$ and the top line segment in the $x$-direction is
\begin{align*}
m_i^j=\frac{[2^{-j-1}+2^{-2(j+1)}]\cdot i}{j}-\frac{[2^{-j-1}+2^{-2(j+1)}]\cdot i}{j+1}.
\end{align*}

It is clear that $m_i^j\approx \frac{i\cdot 2^{-j}}{j^2}$ when $j$ is large.
The distance of the top and the bottom in $y$-direction is $\frac{2}{j^2}$. In Figure~\ref{fig:domain3}, $\textbf{a}=(\frac{2^{-2(j+1)}}{j+1},0)$, $\textbf{p}=(\frac{i\cdot 2^{-j}}{2j^2},\frac{1}{j^2})$ and $\textbf{q}=(\frac{i\cdot 2^{-j}}{2j^2}+\frac{2^{-(j+1)}}{j},\frac{1}{j^2})$. We will write down below a quasiconformal mapping $f_j:A\to A'$ such that $f_j$ maps the bottom line segment of $A$ linearly to $0\textbf{a}$ and the top line segment of $A$ linearly to $\textbf{pq}$, respectively. The line $0\textbf{p}$ is of the form $y=k_1x$, where
\begin{align*}
k_1=\frac{1/j^2}{i\cdot 2^{-j}/(2j^2)}=\frac{2^{j+1}}{i}\geq 1.
\end{align*}
Similarly, the line $\textbf{aq}$ is of the form $y=k_2(x-\frac{2^{-2(j+1)}}{j+1})$, where
\begin{align*}
k_2=\frac{1/j^2}{\frac{i\cdot 2^{-j}}{2j^2}+\frac{ 2^{-(j+1)}}{j}-\frac{2^{-2(j+1)}}{j+1}}\approx \frac{2^j}{i+j}.
\end{align*}

\begin{figure}[h]
  \includegraphics[width=12cm]{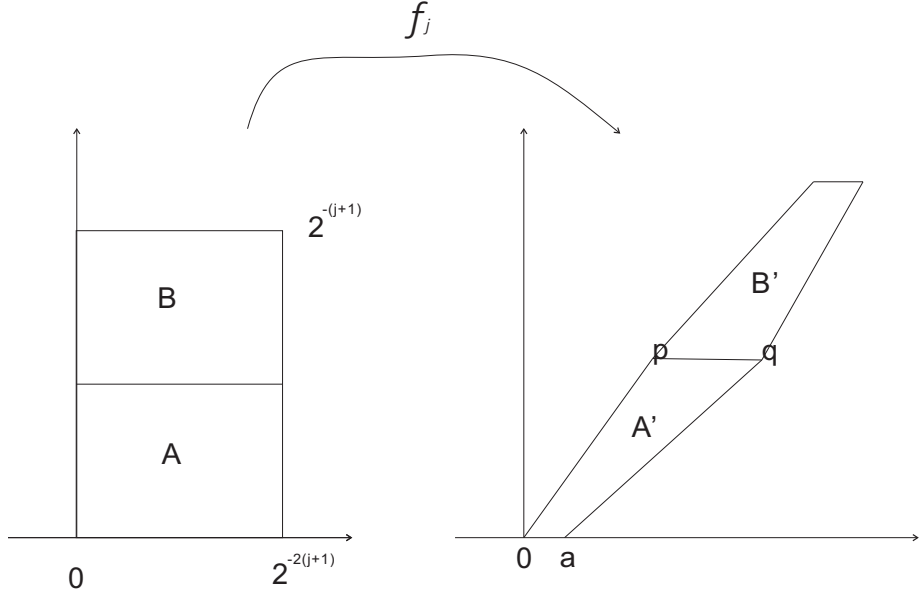}\\
  \caption{The quasiconformal mapping from $\tilde{ Q}_i$ to $Q_i'$}\label{fig:domain3}
\end{figure}

We are looking for a quasiconformal mapping of the form $f_j^i(x,y)=(\tilde{g_j}(y)x+g_j(y),k_1g_j(y))$, where $\tilde{g}_j(y)=k_1 g_j'(y)$ for all $y\in [0,2^{-(j+1)}]$ and $g_j$ is a smooth increasing function. Clearly, such a mapping $f_j$ maps horizontal line segments to horizontal line segments. We further require that it maps the left side of $A$ to $0\textbf{p}$ and the right side of $A$ to $\textbf{aq}$, $g_j(0)=0$, $g_j(2^{-j})=\frac{1}{j^2}$ and $\tilde{g}_j(0)=\frac{1}{j+1}$. By definition,
\begin{align*}
f_j^i(2^{-2(j+1)},y)=(\tilde{g}_j(y)\cdot 2^{-2(j+1)}+g_j(y),k_1g_j(y)).
\end{align*}
The further requirements are satisfied if  $\tilde{g}_j=k_1g_j'$,
\begin{align}\label{eq:1}
g_j(y)=k_2\cdot k_1^{-1}\tilde{g}_j(y)\cdot 2^{-2(j+1)}+\frac{k_2}{k_1}g_j(y)-\frac{k_2}{j}\cdot 2^{-2(j+1)},
\end{align}
\begin{align}\label{eq:2}
g_j(0)=0, g_j(2^{-j})=\frac{1}{j^2}\quad\text{and}\quad \tilde{g}_j(0)=\frac{1}{j+1}.
\end{align}
One can easily solve the above system of equations by setting
$g_j(y)=a\cdot e^{a_j^i y+c}-b$, where
\begin{align*}
a_i^j=2^{2(j+1)}\frac{k_1-k_2}{k_1k_2},\quad b=\frac{1}{k_1(j+1)a_j^i}
\end{align*}
and the constants $b$ and $c$ are chosen such that
\begin{align*}
a\cdot e^c=b\quad \text{and}\quad a\cdot e^{a_j^i 2^{-j}+c}-b=\frac{1}{j^2}.
\end{align*}
We next show that $f_j^i$ is a quasiconformal mapping. A direct computation gives us
\begin{align*}
Df_j^i(x,y)=\begin{bmatrix}
 \tilde{g}_j(y) & \tilde{g}_j'(y)x+g_j'(y) \\
 0 & k_1g_j'(y)
\end{bmatrix}.
\end{align*}
We only need to show that $\tilde{g}_j'(y)x+g_j'(y)\leq  Mk_1g_j'(y)$, for some constant $M$ independent of $i$ and $j$, and for all $x, y\in A$. Since $k_1\geq 1$, it suffices to bound $\tilde{g}_j'(y)x$.
By definition,
\begin{align*}
\tilde{g}_j(y)=k_1g_j'(y)=k_1aa_j^ie^{a_j^iy+c}
\end{align*}
and
\begin{align*}
\tilde{g}_j'(y)=k_1a_j^ig_j'(y).
\end{align*}
Hence we only need to find a uniform bound on $x\cdot a_j^i$. For this, we first note that $k_2$ is bounded from below by $\frac{1}{2}$ and $\frac{k_1-k_2}{k_1}\leq 1$.  Since $x\in [0,2^{-2(j+1)}]$, we have
\begin{align*}
a_j^ix\leq \frac{k_1-k_2}{k_1k_2}\cdot 2^{2(j+1)}x\leq 2.
\end{align*}

This implies that $\tilde{g}_j'(y)x+g_j'(y)\leq 3k_1g_j'(y)$ and so $f_j^i$ is quasiconformal. Notice that $f_j^i(x,0)=(\frac{x}{j+1},0)$, so that, after suitable translations, $f_j^i$ matches with our scaling on the top of $\Omega_{j+1}$. In a similar manner, one can write down a quasiconformal mapping from $B$ to $B'$ such that it coincides with $f_j^i$ on $\textbf{pq}$ and is linear on each line segment. In fact, the quasiconformal mapping just slightly differs from the reflection of $f_j$ with respect to the line segment $\textbf{pq}$ (since the length of $0\textbf{a}$ is approximately the same as the length of the top line segment when $j\to \infty$ and the picture is exactly a reflection with respect to $\textbf{pq}$). When a suitable coordinate system is fixed, it is clear that the mappings $f_j^{i_1}$ and $f_j^{i_2}$ only differ by a translation in $x$-direction and hence the desired global quasiconformal mapping $f_j$ from $\Omega_j$ to $\Omega_j'$ follows by gluing all $f_j^i$'s and the scaling maps.

In this manner, the domain $\Omega'$ is well-defined. We can define the quasiconformal mapping $g:\Omega'\to \Omega$ by setting $g|_{\Omega_j'}=f_j^{-1}$. Moreover, $g$ cannot be uniformly continuous since for each $j\in \mathbb{N}$, it maps a rectangle of length $\frac{1}{j}$ linearly to a rectangle of length $1$.

\end{proof}

\begin{proof}[Proof of Corollary~\ref{coro:co-example to hk}]
Let $\Omega'$ and $\Omega$ be the domains given in the proof of Theorem~\ref{example:sharpness}. Let $x_0$ be the point marked in Figure~\ref{fig:domain1}. It is easy to check that the assumption~\eqref{eq:QH growth} is satisfied and hence the claim follows.
\end{proof}

\begin{proof}[Proof of Theorem~\ref{example:sharpness2}]
We will give the detailed constructions of our domains and quasiconformal mapping for $n=3$ and indicate how to pass it to all dimensions at the end of the proof. The idea of the 3-dimensional construction is similar to the one above and we simply fatten the ``$\Omega_0$" part of the planar domain in Figure~\ref{fig:domain1} along the third direction; see Figure~\ref{fig:domain7} below. 

\begin{figure}[h]
  \includegraphics[width=8cm]{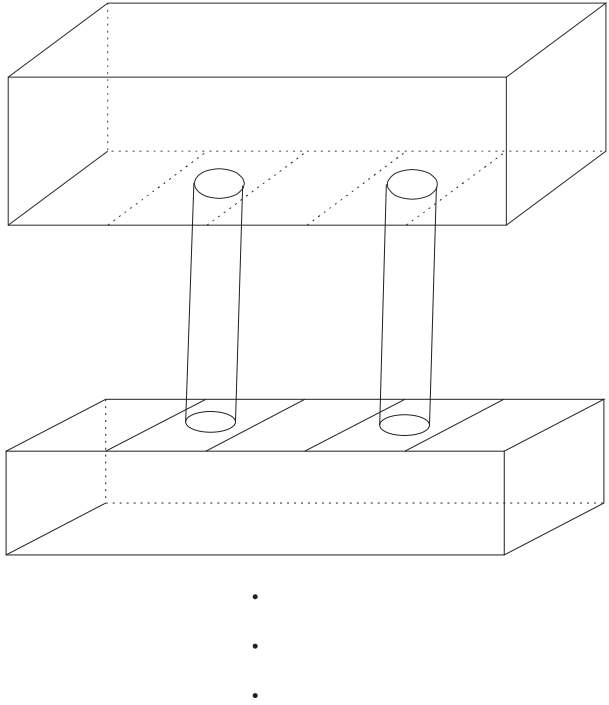}\\
  \caption{The first part of our domain $\Omega$} \label{fig:domain7}
\end{figure}

The top part of Figure~\ref{fig:domain7} consists of a rectangle of length 1, width $\frac{1}{2^2}$ and height $\frac{1}{2^3}$. In the bottom, the rectangle has length 1, width $\frac{1}{2^4}$ and height $\frac{1}{2^6}$. We attach four cylindrical ``legs" of height $2\cdot 2^{-2}$ between these rectangles. The radius of the cylinder is about $2^{-3}$ and the distance between them is about $2^{-2}$.

We can proceed our construction in the following manner. At step $j$, the top part consists of a rectangle of length 1, width $2^{-2j}$ and height $2^{-3j}$. In the bottom, the rectangle has length 1, width $2^{-2(j+1)}$ and height $2^{-3(j+1)}$. We attach $2^{2j}$ equi-distributed cylindrical ``legs" of height $2^{-2j}$ between them. The radius of the cylinder is about $2^{-3j}$ and the distance between two consecutive cylinders is about $h_j=j\cdot 2^{-2j}$. It is clear from our construction that $\Omega$ is an $s$-John domain for any $s\in (1+\frac{1}{2},\infty)$.

Our source domain $\Omega'$ is obtained by a similar scaling procedure as in the proof of Theorem~\ref{example:sharpness}. To be more precise, at step $j$, we scale the top rectangle by $\frac{1}{j^2}$ and replace the associated $2^j$ cylindrical ``legs" by the same number of new ``legs". The vertical distance between the scaled top rectangle and the bottom rectangle is set to be $h_j'=\frac{2}{j^2}$.

\begin{figure}[h]
  \includegraphics[width=10cm]{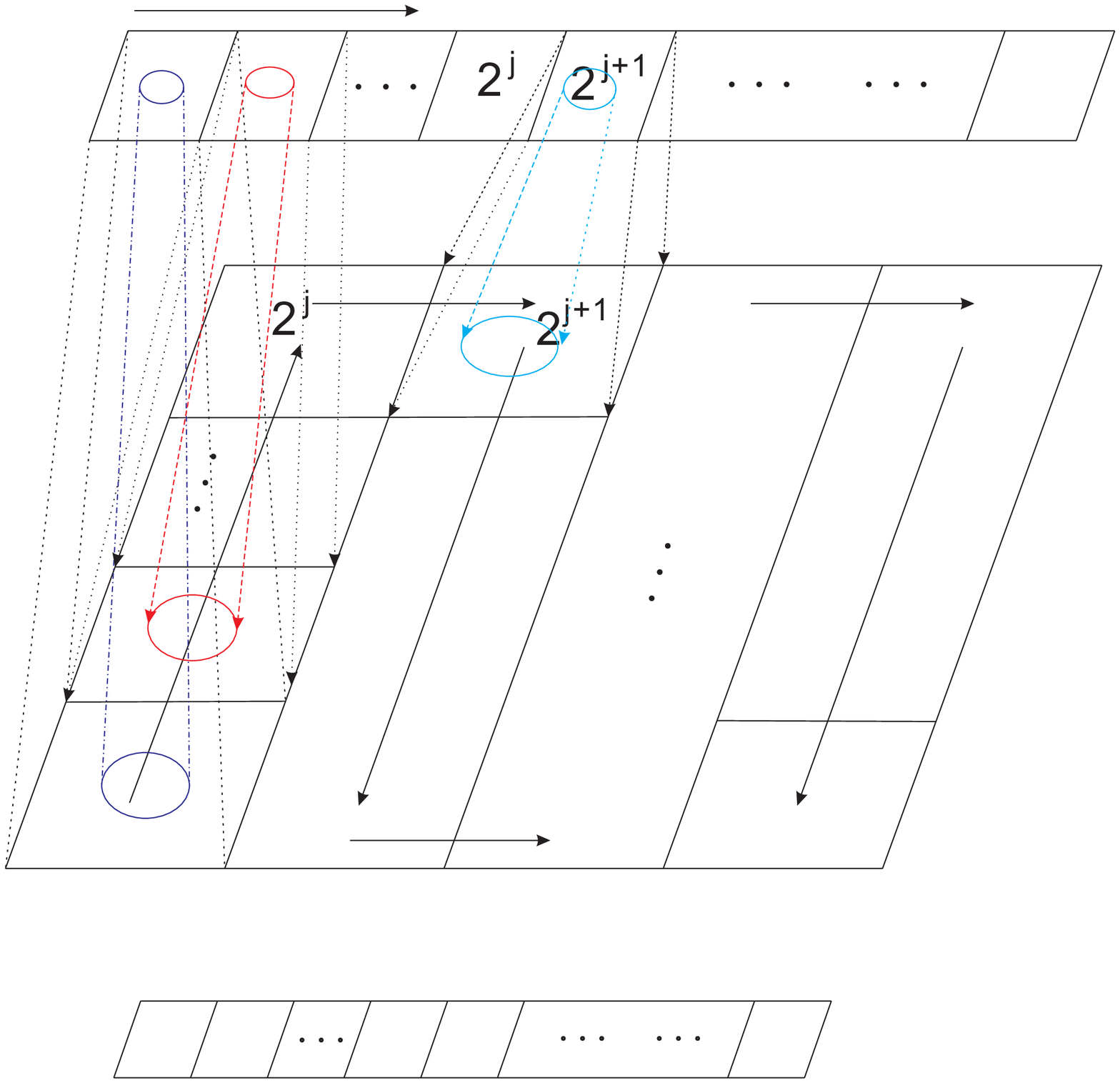}\\
  \caption{The new ``legs" at step $j$}\label{fig:domain10}
\end{figure}

We next explain how to select the new ``legs", see Figure~\ref{fig:domain10} for a top view. In Figure~\ref{fig:domain10}, the top rectangle has length $\frac{1}{j^2}$ and width $\frac{2^{-2j}}{j^2}$. It consists of $2^{2j}$ squares of side-length $\frac{2^{-2j}}{j^2}$. The bottom rectangle has length $\frac{1}{(j+1)^2}$ and width $\frac{2^{-2(j+1)}}{(j+1)^2}$. The vertical distance between these rectangles is $h_j'$.  We insert a square $S_j$ of side-length $\frac{1}{j^2}$ in the middle of the two rectangles, \ie the (vertical) distance between $S_j$ and either of the rectangles is $\frac{1}{j^2}$. We divide $S_j$ into $2^{2j}$ subsquares of side-length $\frac{2^{-j}}{j^2}$. Next, we set up a one-to-one correspondence between the $2^{2j}$ squares in the top rectangle and the subsquares in $S_j$. To be more precise, we first construct $2^{2j}$ affine ``rectangles" between each square in the top rectangle and each subsquare in $S_j$ and then we insert a ``cylindrical leg" inside each affine ``rectangle", see Figure~\ref{fig:domain10} for the order of the affine ``rectangles". The radius of the top circle of the ``cylindrical leg" is set to be $\frac{2^{-3j}}{j^2}$ and the radius of the bottom circle is $\frac{2^{-j}}{j^2}$. Since the $2^{2j}$ affine ``rectangles" have disjoint interiors, the $2^{2j}$ ``cylindrical legs" are pairwise disjoint. As in the proof of Theorem~\ref{example:sharpness}, we use a similar construction between $S_j$ and the bottom rectangle.

Reasoning as in the proof of Theorem~\ref{example:sharpness}, we only need to write down quasiconformal mappings between these ``legs". Note that our construction implies that all the $2^{2j}$ ``cylindrical legs" are bi-Lipschitz equivalent, with a constant independent of $j$. So finally we reduce the problem to the existence of a quasiconformal mapping $g$ as in Figure~\ref{fig:domain6}.  
	
\begin{figure}[h]
  \includegraphics[width=10cm]{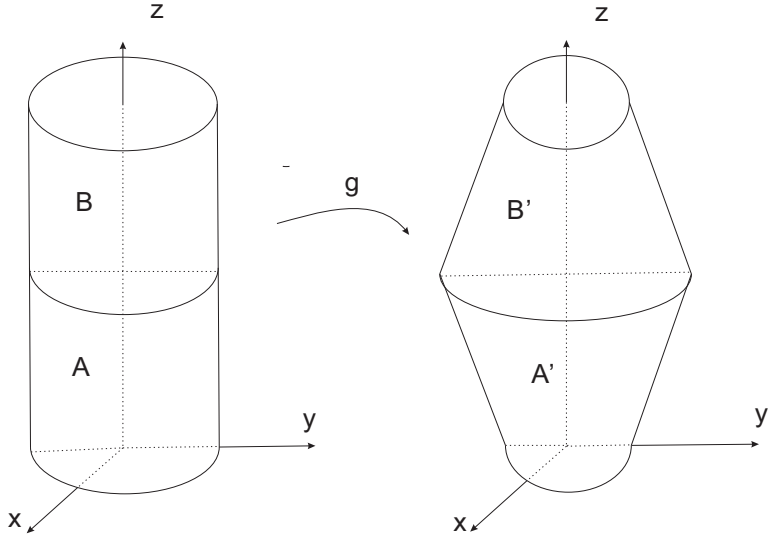}\\
  \caption{The quasiconformal mapping from a ``cylinder" to a  ``double cone"}\label{fig:domain6}
\end{figure}	
	
We will use the coordinate system marked in Figure~\ref{fig:domain6} and write down a quasiconformal mapping $g$ from $A$ onto $A'$ such that $g$ is a scaling between the bottom disks. Set
\begin{equation*}
g(x,y,z)=(g_1(z)x,g_1(z)y,g_2(z)).
\end{equation*}
We require that $g_1(0)=\frac{1}{j^2}$, $g_1(h_j)=\frac{2^{2j}}{j^2}$, $g_2(0)=0$, $g_2(h_j)=h_j'$ and $g_2'(z)=g_1(z)$ for all $z\in [0,h_j]$. It is easy to check that with these requirements, $g$ will be a quasiconformal mapping that maps $A$ to $A'$ such that $g$ is the desired scaling between the bottom disks. One can use the map $g_2$ of the form $g_2(z)=a_j(e^{b_jz}-1)$, where $a_j\approx \frac{2^{-2j}}{j^2}$ and $b_j\approx 2^{2j}$. We leave the simple verification to the interested readers.

As in the planar case, the global quasiconformal mapping $f:\Omega'\to\Omega$ is obtained by gluing all these $g'$s and the corresponding scaling mappings.
Moreover, reasoning as in the planar case, we can easily conclude that $f$ cannot be uniformly continuous with with respect to the metrics $d(x,y)=|x-y|$ in $\Omega$ and $d_I$ in $\Omega'$. 

The construction of the general $n$-dimensional case can be proceeded in a similar manner. In step $j$, $\Omega_j$ consists of a $n$-dimensional rectangle of length $a_1=1$ and (other) edge-lengths $a_2=\cdots=a_{n-1}=2^{-(n-1)j}$, $a_n=2^{-nj}$ and $2^j$ ``cylindrical legs" of length $h_j=j\cdot 2^{-(n-1)j}$. The radius of the cylinder is $2^{-nj}$. So $\Omega$ is an $s$-John domain for any $s\in (1+\frac{1}{n-1},\infty)$.

The source domain $\Omega'$ is obtained by a similar scaling procedure as before. To be more precise, at step $j$, we scale the top rectangle by $\frac{1}{j^2}$ and replace the associated $2^j$ cylindrical ``legs" by the same number of new ``legs". The vertical distance between the scaled top rectangle and the bottom rectangle is set to be $h_j'=\frac{2}{j^2}$.

We use a similar idea as before to obtain new ``legs" between the top rectangle and bottom rectangle as in Figure~\ref{fig:domain10}. Namely, we insert a $(n-1)$-dimensional cube of edge-length $\frac{1}{j^2}$ and then divide it into $2^{(n-1)j}$ subcubes of edge-length $\frac{2^{-j}}{j^2}$. Then attach $2^{(n-1)j}$ affine ``rectangles" in a similar manner as before. Inside each affine ``rectangle", we insert a ``cylindrical leg". The radius of the top of the ``cylindrical leg" is $\frac{2^{-nj}}{j^2}$ and the radius of the bottom is $\frac{2^{-j}}{j^2}$. Reasoning as before, one essentially only needs to write down a quasiconformal mapping $g$ between these ``legs". 

The global quasiconformal mapping $f:\Omega'\to\Omega$ is obtained by gluing all these $g'$s and the corresponding scaling mappings.
Moreover, reasoning as in the planar case, we can easily conclude that $f$ cannot be uniformly continuous with with respect to the metrics $d(x,y)=|x-y|$ in $\Omega$ and $d_I$ in $\Omega'$.

\end{proof}

%
%
%

\begin{proof}[Proof of Corollary~\ref{coro:co-example to hk2}]
Let $\Omega'$ and $\Omega$ be the domains given in the proof of Theorem~\ref{example:sharpness2}. Let $x_0$ be the central point in the first rectangle of $\Omega$. It is easy to check that the assumption~\eqref{eq:QH growth2} is satisfied and hence the claim follows.
\end{proof}

\begin{proof}[Proof of Theorem~\ref{example:sharp3}]

The idea of the construction is similar to that used in the proof of Theorem~\ref{example:sharpness}. When $n=2$ and $s>2$, our $s$-John domain $\Omega\subset\bR^2$ is the same as that in Figure~\ref{fig:domain1} except possible differences in the parameters. When $n\geq 3$ and $s>2$, we simply fatten the planar picture. In the following, we point out the difference of the construction of the $s$-John domain in the planar case and give a more detailed construction of the 3-dimensional analog, while indicating the general construction in the end of the proof.

We first consider the case $n=2$. Fix $s\in (2,\infty)$. Let $\Omega$ be the domain given as in Figure~\ref{fig:domain1} with $a_j=2^{-(j+1)}$, $b_j=2^{-(j+1)}$ and $c_j=2^{-(j+1)s}$. Then $\Omega$ is an $s$-John domain. The domain $\Omega'$ is constructed in a similar fashion and finally we need a quasiconformal mapping from $\tilde{Q}_i$ onto $Q_i'$ as in Figure~\ref{fig:domain3}. The only difference from Figure~\ref{fig:domain3} is that instead of $\frac{2^{-j}}{j}$, we set $d(p,q)=\frac{2^{-j(s-1)}}{j}$. One can check the desired quasiconformal mapping $f$ is of the same form as in Figure~\ref{fig:domain3} with the obvious replacement of parameters. Clearly $f$ is not uniformly continuous with respect to the metrics $d(x,y)=|x-y|$ in $\Omega$ and $d_I$ in $\Omega'$.

We next verify that  $\Omega$ satisfies the Gehring-Hayman inequality, \ie we need to show that
\begin{align}\label{eq:desired}
l([x,y])\leq Cd_I(x,y)
\end{align}
for all $x,y\in \Omega$. We first consider the case that $x,y\in \Omega_j$ for some $j\in \mathbb{N}$. 
Recall that the $\Omega_j$-part consists of a rectangle of $Q_j$ and $2^j$ rectangular ``legs" $Q_{ji}$, 
$i=1,2,\dots,2^j$. If both $x$ and $y$ lie in the rectangle $Q_j$ or both lie in some ``leg" $Q_{ji}$,
 then the length of $[x,y]$ essentially equals to the length of the line segment $\overline{xy}$ that connects $x$ and $y$ 
 and so~\eqref{eq:desired} holds.  If $x\in Q_{jk}$ and $y\in Q_{jl}$, $k,l\in \{1,2,\dots,2^j\}$ and
  $k<l$, then $[x,y]$ can be essentially written as $\gamma_{xz}\cup \overline{zw}\cup\gamma_{wy}$, 
  where $z$ is a point on the (horizontal) core line segment of $Q_j$ with the same first 
  coordinate as $x$ such that $\gamma_{xz}$ is essentially the Euclidean geodesic 
  $\overline{xz}$ and $w$ is a point on the (horizontal) core line segment of $Q_j$ 
  with the same first coordinate as $y$ such that $\gamma_{wy}$ is essentially the 
  Euclidean geodesic $\overline{wy}$. Since $d_I(x,y)$ is comparable to the
   difference of $x$ and $y$ in horizontal directions,  we easily obtain~\eqref{eq:desired}. 
   If $x\in Q_{ji}$ and $y\in Q_j$, then $[x,y]$ can be essentially written as 
  $\gamma_{xz}\cup \overline{zy}$, where $z$ is a point on the (horizontal) 
  line segment of $Q_j$ with the same first coordinate as $x$ and the same second 
  coordinate as $y$ such that $\gamma_{xz}$ is essentially the Euclidean geodesic 
  $\overline{xz}$. Thus~\eqref{eq:desired} holds in this case as well. Similar arguments 
  applies for the case when $x\in \Omega_j$ and $y\in \Omega_{l}$ with $|j-l|=1$.

Now we may assume that $x\in \Omega_j$ and $y\in \Omega_l$, $j-1> l\in \mathbb{N}$. 
Consider first the case $x\in Q_j$ and $y\in Q_l$. We may assume that the first
 coordinate of $x$ is smaller than or equal to the first coordinate of $y$. 
 Let $Q_{j-1,i}$ be the nearest ``leg" (in $\Omega_{j-1}$) to $x$. Then the
  quasihyperbolic geodesic $[x,y]$ goes through $Q_{j-1,i}$, then follows the
   closest ``leg" in $\Omega_{j-2}$ until it reaches $Q_{l}$, and then 
   essentially goes along the Euclidean geodesic to $y$. We may write 
   $[x,y]$ as $\cup_{k=l}^j [x,y]_k$, where $[x,y]_k=[x,y]\cap \Omega_k$.
    Our construction of $\Omega$ implies that there exists  a curve 
    $\beta_{xy}$ such that $d_I(x,y)=l(\beta_{xy})$ and that 
    $|l(\beta_{xy}^k)-l([x,y]_k)|\leq 2^{-k+1}$, $k=l,\dots,j$,
     where $\beta_{xy}^k=\beta_{xy}\cap \Omega_k$. Notice that
\begin{align*}
l([x,y])=l_1+l(\cup_{l<k<j}[x,y]_k)+l_2,
\end{align*}
where $l_1=l([x,y]_j)$ and $l_2=l([x,y]_l)$. It is clear that
\begin{align*}
l(\cup_{l<k<j}[x,y]_k)\approx \sum_{l<k<j}(2^{-ks}+2^{-k})\approx 2^{-(l-1)}.
\end{align*}
Similarly,
\begin{align*}
l(\beta_{xy})=l_1'+l(\cup_{l<k<j}\beta_{xy}^k)+l_2',
\end{align*}
where $l_1'=l(\beta_{xy}^j)$ and $l_2'=l(\beta_{xy}^l)$ and
\begin{align*}
l(\cup_{l<k<j}\beta_{xy}^k)\approx \sum_{l<k<j}(2^{-ks}+2^{-k})\approx 2^{-(l-1)}.
\end{align*}
In this case, since $l_1<2^{-(l-1)}$,  we may assume that $l_2\geq   2^{-l+2}$. 
However, since $|l_2-l_2'|\leq 2^{-l+1}$, we obtain that
\begin{align*}
2l_2'\geq 2(l_2-2^{-l+1})\geq l_2.
\end{align*}
Thus~\eqref{eq:desired} holds in this case as well. The other cases can be proved via a
similar argument. Therefore, we have verified~\eqref{eq:desired} in $\Omega$.
 
To verify that $\Omega'$ satisfies the Gehring-Hayman inequality is more complicated. We first consider
the case $x',y'\in \Omega_j'$ for some $j\in \mathbb{N}$. Recall that the $\Omega_j'$-part consists 
of a rectangle of $Q_j'$ and $2^j$ cone-like ``legs" $Q_{ji}'$,
$i=1,2,\dots,2^j$. If both $x'$ and $y'$ lie in the rectangle $Q_j'$ or both lie in some ``leg" 
$Q_{ji}'$,
 then the length of $[x',y']$ essentially equals to the length of the line segment $\overline{x'y'}$ that connects $x'$ and $y'$
 and so the Gehring-Hayman inequality holds. If $x'\in Q_{jk}'$ and $y'\in Q_{jl}'$, 
 $k,l\in \{1,2,\dots,2^j\}$ and $k<l$, then there are two different cases for the geodesic $\gamma'$ that 
 connects $x'$ and $y'$ in $\Omega'$: either $\gamma'$ first goes up from $x'$, passes through $Q_j'$ and then goes
 down to $y'$ or first goes down from $x'$, passes through $Q_{j+1}'$, and then goes up to $y'$. Essentially 
 we have to deal with two cases, either both $x'$ and $y'$ are close to $Q_{j}'$ or both $x'$ and $y'$ are 
 close to $Q_{j+1}'$.  For this, we need the following two basic facts:  firstly, if $x',y'\in Q_j'$, then $k_{\Omega'}(x',y')=k_\Omega(f(x'),f(y'))$. Secondly, if $x',y'\in Q_{ji}'$, then $k_{\Omega'}(x',y')\approx k_\Omega(f(x'),f(y'))$. For example, if both $x'$ and $y'$ are close to $Q_{j}'$, then the quasihyperbolic geodesic $[x',y']$ has to go up from $x'$, pass through $Q_j'$ and then go down to $y'$, since otherwise, the quasihyperbolic distance will be much bigger. If $x'\in Q_{ji}'$ and $y'\in Q_j'$, then $[x',y']$ goes through $Q_{ji}'$ to $Q_j'$ and then essentially follows an Euclidean geodesic to $y'$. Thus, the Gehring-Hayman inequality holds. Similar arguments 
  applies for the case when $x'\in \Omega_j'$ and $y'\in \Omega_{l}'$ with $|j-l|=1$.
 
Now we may assume that $x'\in \Omega_j'$ and $y\in \Omega_l'$, $j-1> l\in \mathbb{N}$. 
Consider first the case $x'\in Q_j'$ and $y'\in Q_l'$. Let $Q_{j-1,i}'$ be the nearest ``leg" (in $\Omega_{j-1}'$) to $x'$. Then the
  quasihyperbolic geodesic $[x',y']$ goes through $Q_{j-1,i}'$, then follow the
   closest ``leg" in $\Omega_{j-2}'$ until it reaches $Q_{l}'$, and then 
   essentially goes along the Euclidean geodesic to $'y$. We may write 
   $[x',y']$ as $\cup_{k=l}^j [x',y']_k$, where $[x',y']_k=[x',y']\cap \Omega_k'$.
    Our construction of $\Omega'$ implies that there exists  a curve 
    $\beta_{x'y'}$ such that $d_I(x',y')=l(\beta_{x'y'})$ and that 
    $|l(\beta_{x'y'}^k)-l([x',y']_k)|\leq \frac{2}{k^2}$, $k=l,\dots,j$,
     where $\beta_{x'y'}^k=\beta_{x'y'}\cap \Omega_k'$. Notice that
\begin{align*}
l([x',y'])=l_1+l(\cup_{l<k<j}[x',y']_k)+l_2,
\end{align*}
where $l_1=l([x',y']_j)$ and $l_2=l([x',y']_l)$. It is clear that
\begin{align*}
l(\cup_{l<k<j}[x',y']_k)\approx \sum_{l<k<j}(\frac{1}{k^2}+\frac{2^{-k}}{k})\approx \frac{1}{(l-1)^2}.
\end{align*}
Similarly,
\begin{align*}
l(\beta_{x'y'})=l_1'+l(\cup_{l<k<j}\beta_{x'y'}^k)+l_2',
\end{align*}
where $l_1'=l(\beta_{x'y'}^j)$ and $l_2'=l(\beta_{x'y'}^l)$ and
\begin{align*}
l(\cup_{l<k<j}\beta_{x'y'}^k)\approx \sum_{l<k<j}(\frac{1}{k^2}+\frac{2^{-k}}{k})\approx \frac{1}{(l-1)^2}.
\end{align*}
In this case, since $l_1<\frac{2}{(l-1)^2}$,  we may assume that $l_2\geq  \frac{4}{(l-2)^2}$. 
However, since $|l_2-l_2'|\leq \frac{2}{l^2}$, we obtain that
\begin{align*}
2l_2'\geq 2(l_2-\frac{2}{l^2})\geq l_2.
\end{align*}
Thus the Gehring-Hayman inequality holds in this case as well. The other cases can be proved via a
similar argument. Therefore, we have verified the Gehring-Hayman inequality in $\Omega'$.
 

The 3-dimensional construction is similar and we simply fatten the ``$\Omega_0$" part of the planar domain in Figure~\ref{fig:domain1} along the third direction; see Figure~\ref{fig:domain7}.

The top part of Figure~\ref{fig:domain7} consists of a rectangle of length 1, width $\frac{1}{2}$ and height $\frac{1}{2}$. In the bottom, the rectangle has length 1, width $\frac{1}{4}$ and height $\frac{1}{4}$. We attach two cylindrical ``legs" of height $2^{-1}$ between these rectangles. The radius of the cylinder is about $2^{-s}$ and the distance between them is about $2^{-1}$.

We can proceed our construction in the following manner. At step $j$, the top part consists of a rectangle of length 1, width $2^{-j}$ and height $2^{-j}$. In the bottom, the rectangle has length 1, width $2^{-j-1}$ and height $2^{-j-1}$. We attach $2^j$ equi-distributed cylindrical ``legs" of height $2^{-j}$ between them. The radius of the cylinder is about $2^{-js}$ and the distance between two consecutive cylinders is about $2^{-j}$. It is clear from our construction that $\Omega$ is an $s$-John domain.

Our source domain $\Omega'$ is obtained by a similar scaling procedure as in the proof of Theorem~\ref{example:sharpness}. To be more precise, at step $j$, we scale the top rectangle by $\frac{1}{j}$ and replace the associated $2^j$ cylindrical ``legs" by the same number of new ``legs". The vertical distance between the scaled top rectangle and the bottom rectangle is set to be $\frac{2}{j^2}$.

\begin{figure}[h]
  \includegraphics[width=10cm]{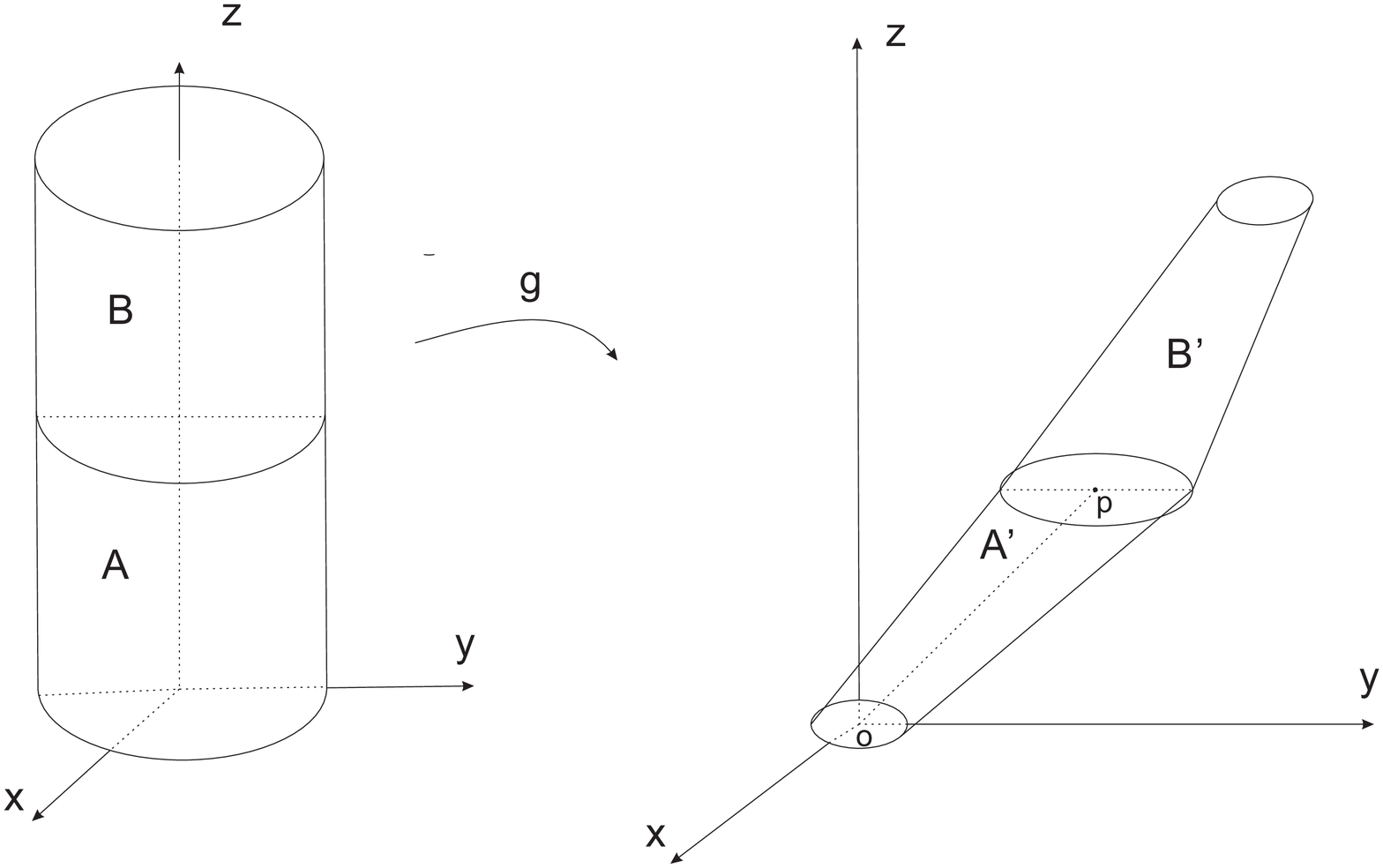}\\
  \caption{A quasiconformal mapping $g$ between ``legs"} \label{fig:domain8}
\end{figure}

The new ``legs" are obtained by rotating the corresponding new ``legs" in Figure~\ref{fig:domain3} along the third direction, see Figure~\ref{fig:domain8}.
Reasoning as in the proof of Theorem~\ref{example:sharpness}, we only need to write down a quasiconformal mappings between these ``legs". For this, we use the coordinate system marked in Figure~\ref{fig:domain8} and look for $g$ of the form
\begin{align*}
g(x,y,z)=(g_1(z)x,g_1(z)y+g_2(z),g_3(z)).
\end{align*}
We require that $g_1(0)=\frac{1}{j}$, $g_2(0)=g_3(0)=0$, $g_1(2^{-j})=\frac{2^j}{j}$, $g_2(2^{-j})\approx \frac{1}{2j^2}$ and $g_3(2^{-j})=\frac{1}{2j^2}$. We can take
\begin{align*}
g_1(z)=a_jb_je^{b_jz},\quad g_2(z)\approx g_3(z)=a_j(e^{b_jz}-1),
\end{align*}
where $a_j\approx 2^{-j}\cdot j^{-2}$ and $b_j\approx j\cdot 2^j$. It is easy to check that $g$ is a quasiconformal mapping with the desired property and we leave the detailed verification to the interested author.

As in the planar case, the global quasiconformal mapping $f:\Omega'\to\Omega$ is obtained by gluing all these $g'$s and the corresponding scaling mappings.
Moreover, reasoning as in the planar case, we can easily conclude that $f$ cannot be uniformly continuous with with respect to the metrics $d(x,y)=|x-y|$ in $\Omega$ and $d_I$ in $\Omega'$. Following the arguments used in the planar case, we easily deduce that both $\Omega'$ and $\Omega$ satisfies the Gehring-Hayman inequality.
 
The construction of the general $n$-dimensional case can be proceeded in a similar manner. In step $j$, $\Omega_j$ consists of a $n$-dimensional rectangle of length $a_1=1$ and (other) side-lengths $a_2=\cdots=a_n=2^{-j}$ and $2^j$ ``cylindrical legs". The radius of the cylinder is $2^{-js}$. So $\Omega$ is an $s$-John domain.
The source domain $\Omega'$ is obtained by a similar scaling procedure as before. To be more precise, at step $j$, we scale the top rectangle by $\frac{1}{j}$ and replace the associated $2^j$ cylindrical ``legs" by the same number of new ``legs". The vertical distance between the scaled top rectangle and the bottom rectangle is set to be $\frac{2}{j^2}$.

The new ``legs" are obtained by rotating the corresponding new ``legs" in Figure~\ref{fig:domain3} along the $n$-th direction, see Figure~\ref{fig:domain8}.
Reasoning as in the proof of Theorem~\ref{example:sharpness}, we only need to write down a quasiconformal mappings between these ``legs". We leave the remaining verifications to the interested readers.
\end{proof}


\end{document}